\documentclass[a4paper,leqno]{amsart}
\usepackage{amssymb}
\usepackage[hyphens]{url} 
\usepackage[utf8]{inputenc}
\usepackage{amsmath,amsfonts,amssymb}
\usepackage{enumerate}
\usepackage{amsthm}
\usepackage{tikz, xcolor, fancyhdr}
\usepackage[colorlinks]{hyperref}
\hypersetup{
  urlcolor     = blue, 
  linkcolor    = blue, 
  citecolor   = blue   
  }
  
\def\R{{\mathbb R}}

\usetikzlibrary{decorations.pathreplacing}

\theoremstyle{plain}
\newtheorem{theorem}{Theorem}[section]

\newtheorem{lemma}[theorem]{Lemma}

\DeclareMathOperator{\supp}{supp}
\DeclareMathOperator{\dist}{dist}

\renewcommand{\leq}{\leqslant}
\renewcommand{\geq}{\geqslant}

\title[counterexample to the off-testing condition in two dimensions]{counterexample to the off-testing condition in two dimensions}



\author[Christos Grigoriadis ]{ Christos Grigoriadis} 

\address{
Department of Mathematics \\ 
Michigan State University   \\ 
East Lansing, Michigan}
\email{grigori4@msu.edu}

\author[ Michail Paparizos]{ Michail Paparizos }
\address{ 
Department of Mathematics \\ 
Michigan State University   \\ 
East Lansing, Michigan}
\email{paparizo@msu.edu}

\begin{document}

{\begin{flushleft}\baselineskip9pt\scriptsize
\end{flushleft}}
\vspace{18mm} \setcounter{page}{1} \thispagestyle{empty}

\maketitle

\begin{abstract}
In proving the local $T_b$ Theorem for two weights in one dimension \cite{SaShUT} Sawyer, Shen and Uriarte-Tuero used a basic theorem of Hyt\"{o}nen \cite{Hy} to deal with estimates for measures living in adjacent intervals. Hyt\"{o}nen's theorem states that the off-testing condition for the Hilbert transform is controlled by the Muckenhoupt's $A_2$ and $A^*_2$ conditions. So in attempting to extend the two weight $T_b$ theorem to higher dimensions, it is natural to ask if a higher dimensional analogue of Hyt\"{o}nen's theorem holds that permits analogous control of terms involving measures that live on adjacent cubes. In this paper we show that it is not the case even in the presence of the energy conditions used in one dimension \cite{SaShUT}. Thus, in order to obtain a local $T_b$ theorem in higher dimensions, it will be necessary to find some substantially new arguments to control the notoriously difficult nearby form. More precisely, we show that Hyt\"{o}nen's off-testing condition for the two weight fractional integral and the Riesz transform inequalities is not controlled by Muckenhoupt's $A_2^\alpha$ and $A_2^{\alpha,*}$ conditions and energy conditions.
\end{abstract}

\maketitle

\section{Introduction}
 Characterizing two-weight norm inequalities for singular integrals is an important, long standing open problem, only recently solved in one dimension by Lacey, Sawyer, Shen and Uriarte-Tuero in a two-part paper \cite{LaSaShUT}-\cite{La}. Hyt\"{o}nen \cite{Hy} later removed a technical hypothesis, and for his proof an important piece was to bound the bilinear form when two functions are supported on disjoint half-lines in terms only of (his variant of) the Muckenhoupt $A_2$ constants. As mentioned in the abstract, Sawyer, Shen and Uriarte-Tuero \cite{SaShUT} used Hyt\"{o}nen's theorem to estimate the difficult nearby form in the one-dimensional local $T_b$ theorem and it seems natural to ask whether a higher dimensional analogue of Hyt\"{o}nen's theorem is true in order to estimate the nearby form in the higher dimensional local $T_b$ Theorem. Our paper answers this question negatively, even if we assume the energy conditions $\mathcal{E}^\alpha_2, \mathcal{E}^{\alpha,*}_2$, as in the case of the one dimensional two-weighted local $T_b$ Theorem.

 The key idea is the construction of two measures on the plane placed close to each other (Figure 1) so that the off-testing condition fails but the $A_2$ and energy conditions hold using some one-dimensional results from \cite{LaSaUr}. Following closely the aforementioned work, we first construct two measures in $\mathbb{R}$ with the novelty being the use of a `wrong' homogeneity of the one-dimensional Riesz, Poisson and fractional integrals that accommodates all $0<\alpha<2$. 

Let $0\leq \alpha<n$. For any locally finite Borel measure $\sigma$, we define the fractional integral on $\mathbb{R}^n$ by
$$
I^a (f\sigma)(x)=\int_{\mathbb{R}^n} \frac{f(y)}{|x-y|^{n-\alpha}}d\sigma(y),\ x\notin\supp(f\sigma) \
$$
for any $f\in L^2(\sigma)$, and the Riesz transforms by
$$
R^\alpha_m (f\sigma)(x)=\int_{\R^n}\frac{(t_m-x_m)f(t)}{|x-t|^{n+1-\alpha}}d\sigma(t),\quad  x\notin\supp(f\sigma) , \ \ 1\leq m\leq n
$$
where $x=(x_1,\dots ,x_n)$, $t=(t_1,\dots, t_n)$. If $\omega$ is another locally finite Borel measure, we say that the pair of weights $(\sigma,\omega)$ satisfies the fractional Muckenhoupt $\mathcal{A}^{\alpha}_2$ and $\mathcal{A}^{\alpha,*}_2$ conditions in $\R^n$ if 
\begin{equation*}\label{3}
\mathcal{A}^{\alpha}_2\equiv \sup_{Q\in\mathcal{I}}\mathcal{P}^\alpha(Q,\textbf{1}_{Q^c}\sigma)\frac{\omega(Q)}{|Q|^{1-\frac{\alpha}{n}}}<\infty
\end{equation*}
and
\begin{equation*}\label{1}
\mathcal{A}^{\alpha,*}_2\equiv \sup_{Q\in\mathcal{I}}\mathcal{P}^\alpha(Q,\textbf{1}_{Q^c}\omega)\frac{\sigma(Q)}{|Q|^{1-\frac{\alpha}{n}}}<\infty
\end{equation*}
where $\mathcal{I}$ denotes the collection of all cubes $Q$ in $\R^n$ whose sides are parallel to the axes and 
$$
\mathcal{P}^\alpha(Q,\mu)=\int_{\R^n}\bigg(\frac{|Q|^\frac{1}{n}}{(|Q|^\frac{1}{n}+|x-x_Q|)^2}\bigg)^{n-\alpha}d\mu(x),
$$
with $x_Q$ being the center of the cube, is the reproducing Poisson integral. We also say that the pair $(\sigma,\omega)$ 
satisfies the energy (resp. dual energy) condition if
\begin{equation*}
\left( \mathcal{E}_{2}^{\alpha }\right) ^{2}\equiv \sup_{Q=\dot{\cup}Q_{r}}%
\frac{1}{\sigma(Q)}\sum_{r=1}^{\infty }\left( \frac{\mathrm{P}
^{\alpha }\left( Q_{r},\mathbf{1}_{Q}\sigma \right) }{\left\vert
Q_{r}\right\vert ^{\frac{1}{n}}}\right) ^{2}\left\Vert
x-m^\omega_{Q_{r}}\right\Vert _{L^{2}\left( \mathbf{1}_{Q_{r}}\omega \right) }^{2}<\infty
\end{equation*}%
\begin{equation*}
\left( \mathcal{E}_{2}^{\alpha ,\ast }\right) ^{2}\equiv \sup_{Q=\dot{\cup}%
Q_{r}}\frac{1}{\omega(Q)}\sum_{r=1}^{\infty }\left( 
\frac{\mathrm{P}^{\alpha }\left( Q_{r},\mathbf{1}_{Q}\omega \right) }{
\left\vert Q_{r}\right\vert ^{\frac{1}{n}}}\right) ^{2}\left\Vert
x-m^\sigma_{Q_{r}}\right\Vert _{L^{2}\left( \mathbf{1}_{Q_{r}}\sigma \right) }^{2}<\infty
\end{equation*}
where the supremum is taken over arbitrary decompositions of a cube $Q$
using a pairwise disjoint union of subcubes $Q_{r}$, where 
$$
\mathrm{P}^\alpha(Q,\mu)=\int_{\R^n}\frac{|Q|^\frac{1}{n}}{(|Q|^\frac{1}{n}+|x-x_Q|)^{n+1-\alpha}}d\mu(x)
$$
is the standard Poisson integral and
$$
m_I^\mu\equiv\frac{1}{\mu(I)}\int x d\mu(x)=\left\langle \frac{1}{|I|_\mu}\int x_1d\mu(x),...,\frac{1}{|I|_\mu}\int x_nd\mu(x)\right\rangle.
$$

In the one-dimensional setting, Hyt\"{o}nen \cite{Hy} has characterized the restricted bilinear inequality,
\begin{equation}\label{2}
\bigg|\int_{\R\backslash I}\bigg(\int_I\frac{f(y)}{|x-y|}d\sigma(y)\bigg)g(x)d\omega(x)\bigg|\lesssim\mathcal{D}\big|\big|f\big|\big|_{L^2(\sigma)}\big|\big|g\big|\big|_{L^2(\omega)}
\end{equation}
for all intervals $I$, in terms of the Muckenhoupt conditions, namely
$$
\mathcal{D}\approx\sqrt{\mathcal{A}^0_2}+\sqrt{\mathcal{A}_2^{0,*}}
$$
where $\mathcal{D}$ is the best constant in (\ref{2}). In \cite{Hy} this inequality was proved for complementary half-lines, where it was denoted that the passage to an interval and its complement is then routine. In \cite{SaShUT}, Hyt\"{o}nen's characterization was extended to fractional integrals on the line with the same proof. Namely,
$$
\Bigg|\int_{\R\backslash I} \bigg(\int_{I} \frac{f(y)}{|x-y|^{1-\alpha}}d\sigma(y)\bigg)g(x)d\omega(x)\Bigg|\leq \mathcal{D}^\alpha \big|\big|f\big|\big|_{L^2(\sigma)}\big|\big|g\big|\big|_{L^2(\omega)}
$$
and that $\sqrt{\mathcal{A}_2^\alpha}+\sqrt{\mathcal{A}_2^{\alpha,*}}\approx\mathcal{D}^\alpha$, where $\mathcal{D}^\alpha$ is the best constant in the inequality above. 

Define the off-testing constants $\mathcal{T}_{\textit{off},\alpha}$ and $\mathcal{R}_{j,\textit{off},\alpha}$ in $\mathbb{R}^2$ by
\begin{equation*}\label{4}
\mathcal{T}_{\textit{off},\alpha}^2=\sup_Q \frac{1}{\omega (Q)}\int_{\R^2\backslash  Q}\bigg(\int_Q\frac{1}{|x-y|^{2-\alpha}}d\omega (y)\bigg)^2d\sigma (x)
\end{equation*}
\begin{equation*}\label{5}
\mathcal{R}_{m,\textit{off},\alpha}^2=\sup_Q \frac{1}{\omega (Q)}\int_{\R^2\backslash  Q}\bigg(\int_Q\frac{t_m-x_m}{|x-t|^{3-\alpha}}d\omega (t)\bigg)^2d\sigma (x), \quad 1\leq m\leq 2
\end{equation*}
for all cubes $Q \subset \R^2$ whose sides are parallel to the axes. 
\section{Main result}
We show that in two dimensions, we can find a pair of measures such that  $\mathcal{A}^\alpha_2$, $\mathcal{E}^\alpha_2$ and their dual conditions hold, but the off-testing condition fails. Thus, we prove that we cannot extend Hytonen's theorem in \cite{Hy} in higher dimensions. Indeed,  Theorem \ref{Riesz} provides a counterexample to the analogue of Hyt\"{o}nen's theorem in $\mathbb{R}^2$ as the Riesz transforms for $\alpha=0$ are the extensions of the Hilbert transform in higher dimensions.

\begin{theorem} \label{fracint}
For $0\leq \alpha<2$, there exists a pair of locally finite Borel measures $\sigma, \omega$  in $\R^2$ such that the fractional Muckenhoupt $\mathcal{A}^{\alpha}_2, \mathcal{A}^{\alpha,*}_2$ and the energy $\mathcal{E}_2^\alpha$, $\mathcal{E}_2^{\alpha,*}$ constants are finite but the off-testing constant $\mathcal{T}_{\textit{off},\alpha}$ is not.
\end{theorem}
\begin{theorem}\label{Riesz} 
For $0\leq \alpha<2$, there exists a pair of locally finite Borel measures $\sigma, \omega$  in $\R^2$ such that the fractional Muckenhoupt $\mathcal{A}^{\alpha}_2, \mathcal{A}^{\alpha,*}_2$ and the energy $\mathcal{E}_2^\alpha$, $\mathcal{E}_2^{\alpha,*}$ constants are finite but the off-testing constants $\mathcal{R}_{m,\textit{off},\alpha}$ are not.
\end{theorem}

\section{Proofs Of The Theorems}

We begin with the proof of Theorem \ref{fracint}. The proof of Theorem \ref{Riesz} will be very similar and we will only have to deal with the cancellation occurring in the kernel with Lemma \ref{lemma} being useful.

\begin{proof}[Proof of Theorem \ref{fracint}]
First we build two measures in $\R$, generalizing the work done in \cite{LaSaUr}, and then they will be used for our two dimensional construction.
$$
\underline{\textsc{The One-Dimensional Construction}}
$$
Given $0\leq\alpha<2$, choose  $\frac{1}{3}\leq b<1$ such that $\frac{1}{9}\leq \left(\frac{1-b}{2}\right)^{2-\alpha}\leq\frac{1}{3}$. Let $s_0^{-1}=\left(\frac{1-b}{2}\right)^{2-\alpha}$. Recall the middle-$b$ Cantor set $\mathrm{E}_b$ and the Cantor measure $\ddot{\omega}$ on the closed interval $I^0_1=[0,1]$. At the $k$th generation in the construction, there is a collection $\{I_j^k\}_{j=1}^{2^k}$ of $2^k$ pairwise disjoint closed intervals of length $|I_j^k|=\left(\frac{1-b}{2}\right)^k$. The Cantor set is defined by 
$\mathrm{E}_b=\bigcap_{k=1}^{\infty}\bigcup_{j=1}^{2^k}I_j^k$ and the Cantor measure $\ddot{\omega}$ is the unique probability measure supported in $\mathrm{E}$ with the property that is equidistributed among the intervals $\{I_j^k\}_{j=1}^{2^k}$ at each scale $k$, i.e
$$
\ddot{\omega}(I^k_j)=2^{-k},\ \ \ \  k \geq 0, 1 \leq j \leq 2^k.
$$
We denote the removed open middle $b$th of $I_j^k$ by $G^k_j$ and by $\ddot{z}^k_j$ its center. Following closely \cite{LaSaUr}, we define
$$
\ddot{\sigma}=\sum_{k,j}s^k_j\delta_{\ddot{z}^k_j}
$$
where the sequence of positive numbers $s^k_j$ is chosen to satisfy $\displaystyle \frac{s^k_j\ddot{\omega}(I^k_j)}{|I^k_j|^{4-2\alpha}}=1$, i.e.
$$
s^k_j=\left(\frac{2}{s_0^2}\right)^{k}, \ \  k \geq 0,\ 1 \leq j \leq 2^k.
$$
\textsc{The Testing Constant is Unbounded}.
Consider the following operator
$$
\ddot{T}f(x)=\int_\R\frac{f(y)}{|x-y|^{2-\alpha}}dy
$$
Note that
\begin{eqnarray*}\label{testing value}
\ddot{T}\ddot{\omega}(\ddot{z}^k_1)
\!=\!
\int_{I_1^0}\frac{d\ddot{\omega}(y)}{|\ddot{z}_1^k-y|^{2-\alpha}}
\!\geq\!
\int_{I_1^k}\frac{d\ddot{\omega}(y)}{|\ddot{z}_1^k-y|^{2-\alpha}}
\!\geq\!
\frac{\ddot{\omega} (I_1^k)}{\left(\frac{1}{2}\left(\frac{1-b}{2}\right)^{k}\right)^{2-\alpha}}
\!\approx\!
\left(\frac{s_0}{2}\right)^k
\end{eqnarray*}
since $|\ddot{z}_1^k-y|\leq|\ddot{z}_1^k|$ for $y\in I_1^k$ and $\ddot{z}_1^k=\frac{1}{2}(\frac{1-b}{2})^{k}$. Similar inequalities hold for the rest of $\ddot{z}^k_j$. This implies that the following testing condition fails:
\begin{eqnarray}
\int_{I_1^0}\left(\ddot{T}(\mathbf{1}_{I_1^0}\ddot{\omega})(x)\right)^2d\ddot{\sigma}(y)
\gtrsim
\sum_{k=1}^
\infty\sum_{j=1}^{2^k}s^k_j\cdot \left(\frac{s_0}{2}\right)^{2k}
= \label{testinfinity}
\sum_{k=1}^\infty\sum_{j=1}^{2^k}\frac{1}{2^k}=\infty 
\end{eqnarray}
\textsc{The $\ddot{\mathcal{A}}_2$ Condition}.
Let us now define 
$$
\ddot{\mathcal{P}}(I,\mu)=\int_\mathbb{R}\left(\frac{|I|}{\left(|I|+|x-x_I|\right)^2}\right)^{2-\alpha}\!\!\!\!d\mu(x)
$$
and the following variant of the $A_2^\alpha$ condition:
$$
\ddot{\mathcal{A}}_2^\alpha(\ddot{\sigma},\ddot{\omega})=\sup_I
\ddot{\mathcal{P}}(I,\ddot{\sigma})  \cdot
\ddot{\mathcal{P}}(I,\ddot{\omega})
$$
where the supremum is taken over all intervals in $\mathbb{R}$. We verify that $\ddot{\mathcal{A}}_2^\alpha$ is finite for the pair $(\ddot{\sigma},\ddot{\omega})$. The starting point is the estimate

$$
\ddot{\sigma}(I^\ell_r)=\!\!\!\! \sum_{(k,j):\ddot{z}^k_j \in I^\ell_r} \!\!\!\!\! s^k_j
=
\sum_{k=l}^\infty2^{k-\ell}s_j^k
=
2^{-\ell}\sum_{k=l}^\infty \left(\frac{4}{s_0^2}\right)^{\!k}
\!\approx\!
\left(\frac{2}{s_0^2}\right)^{\ell}=s^\ell_r
$$
and from this, it immediately follows,
\begin{equation}\label{varclass_a2}
\frac{\ddot{\sigma}(I^\ell_j)\ddot{\omega}(I^\ell_j)}{|I_j^\ell|^{4-2\alpha}}\approx
\frac{s_j^\ell \ddot{\omega}(I^\ell_j)}{|I_j^\ell|^{4-2\alpha}}=1,\ \text{for }\ell \geq 0,\ 1 \leq j \leq 2^\ell.
\end{equation}
Now from the definition of $\ddot{\sigma}$ we get,
\begin{eqnarray}
\label{poisson sigma}
\ddot{\mathcal{P}}(I^\ell_r,\ddot{\sigma})
&\leq&
\frac{\ddot{\sigma}(I^\ell_r)}{|I^\ell_r|^{2-\alpha}}+
\int_{I_1^0\backslash I^\ell_r}\left(\frac{|I^\ell_r|}{\left(|I^\ell_r|+|x-x_{I^\ell_r}|\right)^2}\right)^{2-\alpha}\!\!\!\!d\ddot{\sigma}(x)\\
&\leq&
\frac{\ddot{\sigma}(I^\ell_r)}{|I^\ell_r|^{2-\alpha}}+
\sum_{m=0}^{\ell}\sum_{k=m}^\infty \frac{2^{k-m}s^k_j\ |I^\ell_r|^{2-\alpha}}{\left(|I^\ell_r|+b\left(\frac{1-b}{2}\right)^m\right)^{4-2\alpha}}\notag\\
&\lesssim &
\frac{\ddot{\sigma}(I^\ell_r)}{|I^\ell_r|^{2-\alpha}}+
\sum_{m=0}^{\ell}\frac{2^{-m}|I^\ell_r|^{2-\alpha}\left(\frac{4}{s_0^2}\right)^m}{\left(b\left(\frac{1-b}{2}\right)^{m-\ell}|I^\ell_r|\right)^{4-2\alpha}}\notag\\
&=&
\frac{\ddot{\sigma}(I^\ell_r)}{|I^\ell_r|^{2-\alpha}}+
\frac{b^{2\alpha-4}}{|I^\ell_r|^{2-\alpha}} \left(\frac{1}{s_0^2}\right)^{\!\!\ell} \sum_{m=0}^{\ell}2^m\notag\\
&\lesssim&
\frac{\ddot{\sigma}(I^\ell_r)}{|I^\ell_r|^{2-\alpha}}+
\frac{s_r^\ell}{|I^\ell_r|^{2-\alpha}} 
\approx
\frac{\ddot{\sigma}(I^\ell_r)}{|I^\ell_r|^{2-\alpha}}\notag
\end{eqnarray}
and using the uniformity of $\ddot{\omega}$,
\begin{eqnarray}
\label{poisson omega}
\ddot{\mathcal{P}}(I^\ell_r,\ddot{\omega})
&\leq&
\frac{\ddot{\omega}(I^\ell_r)}{|I^\ell_r|^{2-\alpha}}
+ \int_{I_1^0\backslash I^\ell_r}\left(\frac{|I^\ell_r|}{\left(|I^\ell_r|+|x-x_{I^\ell_r}|\right)^2}\right)^{2-\alpha}\!\!\!\!d\ddot{\omega}(x)\\
&\leq&
\frac{\ddot{\omega}(I^\ell_r)}{|I^\ell_r|^{2-\alpha}}+
\sum_{k=1}^\ell \frac{|I^\ell_r|^{2-\alpha}\ \ddot{\omega}(I^k_{j_k})}{\left(|I^\ell_r|+b\left(\frac{1-b}{2}\right)^{k-1}\right)^{4-2\alpha}}\notag\\
&\leq&
\frac{\ddot{\omega}(I^\ell_r)}{|I^\ell_r|^{2-\alpha}}+
\sum_{k=1}^\ell \frac{|I^\ell_r|^{2-\alpha}\ \ddot{\omega}(I^k_{j_k})}{\left(b\left(\frac{1-b}{2}\right)^{k-1-\ell}|I^\ell_r|\right)^{4-2\alpha}}\notag
\end{eqnarray}
\begin{eqnarray}
&\lesssim&
\frac{\ddot{\omega}(I^\ell_r)}{|I^\ell_r|^{2-\alpha}}+
\frac{2^{-\ell}}{|I^\ell_r|^{2-\alpha}} =
2\frac{\ddot{\omega}(I^\ell_r)}{|I^\ell_r|^{2-\alpha}}\notag ,
\end{eqnarray}
where $I^k_{j_k}\subset I^{k-1}_t$,  $I^\ell_r\subset I^{k-1}_t $ and $I^k_{j_k}\cap I^\ell_r=\emptyset$, and where all the implied constants in the above calculations depend only on $\alpha$. From (\ref{poisson sigma}), (\ref{poisson omega}) and \eqref{varclass_a2}, we see that
$$
\ddot{\mathcal{P}}(I_r^\ell,\ddot{\sigma}) 
\ddot{\mathcal{P}}(I_r^\ell,\ddot{\omega})\lesssim 1.
$$
Let us now consider an interval $I\subset I_1^0$ and let $A>1$ be fixed. Then, let $k$ be the smallest integer such that $\ddot{z}_j^k\in AI$; if there is no such $k$, then $AI\subsetneqq G_{j}^\ell$, for some $\ell$. We have the following cases:\\
\textbf{Case 1.} Assume that $I\subset AI \subsetneqq  G_j^k \subset I_j^k$. If $|x_I-\ddot{z}^k_j|\leq \dist(x_I, \partial G^k_j)$ then,
\begin{eqnarray}
\label{first case}
&&
\ddot{\mathcal{P}}(I,\ddot{\sigma})\ddot{\mathcal{P}}(I,\ddot{\omega})=  |I|^{4-2\alpha}\int_{I^0_1}\frac{d\ddot{\sigma}(x)}{(|I|+|x-x_I|)^{4-2\alpha}}\int_{I^0_1}\frac{d\ddot{\omega}(x)}{(|I|+|x-x_I|)^{4-2\alpha}}\\
&\lesssim&
|I|^{4-2\alpha}\left(\frac{s^k_j}{|I|^{4-2\alpha}}+\frac{1}{|I^k_j|^{2-\alpha}}\int_{I^0_1\backslash G^k_j}\frac{|I^k_j|^{2-\alpha}d\ddot{\sigma}(x)}{(|I^k_j|+|x-x_{I^k_j}|)^{4-2\alpha}}\right)\frac{\ddot{\mathcal{P}}(I^k_j,\ddot{\omega})}{|I^k_j|^{2-\alpha}}\notag \\
&\lesssim&
\frac{|I|^{4-2\alpha}}{|I^k_j|^{2-\alpha}}\left(\frac{s^k_j}{|I|^{4-2\alpha}}+\frac{\ddot{\sigma}(I_j^k)}{|I^k_j|^{4-2\alpha}}\right)\frac{\ddot{\omega}(I^k_j)}{|I^k_j|^{2-\alpha}}
\lesssim
\frac{\ddot{\sigma}(I_j^k)\ddot{\omega}(I^k_j)}{|I^k_j|^{4-2\alpha}}\approx 1\notag 
\end{eqnarray}
where in the first inequality we used the fact that $|x-x_I|\approx|x-\ddot{z}_j^k|\gtrsim |I_j^k|$ when $x\notin G^k_j$, since $x_I$ is ``close" to the center of $G^k_j$, and for the second inequality we used \eqref{poisson sigma} and \eqref{poisson omega}. 

If $|x_I-\ddot{z}^k_j|> \dist(x_I, \partial G^k_j)$, we can assume 
$
b\left(\frac{1-b}{2}\right)^{m-1}\leq |I|\leq b\left(\frac{1-b}{2}\right)^m
$
for some $m>k$, since for $m=k$ we have $|I|\approx |I^k_j|$, $|x-x_I|\gtrsim|x-x_{I^k_j}|$ for $x \notin G^k_j$ and we can repeat the proof of \eqref{first case}. Now let $I^m_t$ be the $m$-th generation interval that is closer to $I$ that touches the boundary of $G^k_j$. We have, using $|x_{I^m_t}-\ddot{z}^\ell_j|\lesssim|x_I-\ddot{z}^\ell_j|$, for all $\ell \geq 1, 1\leq j\leq 2^\ell$,
$
\ddot{\mathcal{P}}(I,\ddot{\sigma})\lesssim\ddot{\mathcal{P}}(I^m_t,\ddot{\sigma})
$ and $\ddot{\mathcal{P}}(I,\ddot{\omega})\lesssim\ddot{\mathcal{P}}(I^m_t,\ddot{\omega})$, which imply 
$$
\ddot{\mathcal{P}}(I,\ddot{\sigma})\ddot{\mathcal{P}}(I,\ddot{\omega})\lesssim 1.
$$

\textbf{Case 2.} Now assume $G_j^k\subset AI$. If $I^k_j \cap I=\emptyset$, then, using the minimality of $k$, $I \subset G^m_t$ for some $m<k$ and we can repeat the proof of \eqref{first case}. If $I^k_j \cap I \neq \emptyset$ then $|I| \lesssim |I^k_j| $ since otherwise $AI$ would contain $\ddot{z}^{k-1}_t$, contradicting the minimality of k if we fix $A$ big enough depending only on $\alpha$. Hence we have:
$$
|G_j^k|+|x-\ddot{z}_j^k|\leq |G_j^k|+|x_I-\ddot{z}_j^k|+|x-x_I|\leq 
\left(A+\frac{A}{2}\right)|I|+|x-x_I|
$$
which implies that
$$
\ddot{\mathcal{P}}(I,\ddot{\sigma})               \!\lesssim\!
\int_{I_1^0}\frac{|I|^{2-\alpha}}{\left(|G_j^k|+|x-\ddot{z}_j^k|\right)^{4-2\alpha}}
d\ddot{\sigma}(x) \!\lesssim\!
\frac{|I|^{2-\alpha}}{|I_{j}^k|^{2-\alpha}} 
\!\int_{I_1^0}\!\frac{|I_j^k|^{2-\alpha}}{\left(|I_j^k|+|x-\ddot{z}_j^k|\right)^{4-2\alpha}}
d\ddot{\sigma}(x)
$$
and similarly
$$
\ddot{\mathcal{P}}(I,\ddot{\omega})                    \lesssim 
\frac{|I|^{2-\alpha}}{|I_{t}^k|^{2-\alpha}} 
\ddot{\mathcal{P}}(I_{j}^k,\ddot{\omega})                \leq
\ddot{\mathcal{P}}(I_{j}^k,\ddot{\omega}) .
$$
which implies
$$
\ddot{\mathcal{P}}(I,\ddot{\sigma})   \ddot{\mathcal{P}}(I,\ddot{\omega})   \lesssim 1
$$
\textbf{Case 3.} If neither $G_j^k\cap AI\neq G_j^k$ nor $G_j^k\cap AI\neq AI$, note that $G_j^k\subset 3AI$ and we repeat again the proof of Case 2. 

Thus, for any interval $I\subset I_1^0$, we have shown that $\ddot{\mathcal{P}}(I,\ddot{\sigma})\ddot{\mathcal{P}}(I,\ddot{\omega})\lesssim 1$, which implies
\begin{equation}\label{var_a2_bdd}
\ddot{\mathcal{A}_2^{\alpha}}(\ddot{\sigma},\ddot{\omega})<\infty
\end{equation}

\textsc{The Energy Constants $\ddot{\mathcal{E}}$ and $\ddot{\mathcal{E}}^{*}$}.
Now define the following variant of the energy constants
\begin{eqnarray*}
\ddot{\mathcal{E}} &=& \sup_{I=\dot{\bigcup}I_r} \frac{1}{\ddot{\sigma}(I)} \sum_{r\geq 1} \ddot{\omega}(I_r)E(I_r,\ddot{\omega})^2\ddot{\mathrm{P}}(I_r,\mathbf{1}_I\ddot{\sigma})^2\\
\ddot{\mathcal{E}}^{*} &=& \sup_{I=\dot{\bigcup}I_r} \frac{1}{\ddot{\omega}(I)} \sum_{r\geq 1} \ddot{\sigma}(I_r)E(I_r,\ddot{\sigma})^2\ddot{\mathrm{P}}(I_r,\mathbf{1}_I\ddot{\omega})^2
\end{eqnarray*}
where the supremum is taken over the different intervals $I$ and all the different decompositions of $I=\dot{\bigcup}_{r\geq 1}I_r$, and 
$$
\ddot{\mathrm{P}}(I,\mu)=\int_\mathbb{R}\frac{|I|}{\left(|I|+|x-x_I|\right)^{3-\alpha}}d\mu(x),
$$
$$
E(I,\mu)^2=\frac{1}{2}\frac{1}{\mu(I)^2}\int_{I}\int_{I} \frac{(x-x')^2}{|I|^2} d\mu(x')d\mu(x)=\frac{1}{\mu(I)}\cdot \left\Vert x-m_I^\mu \right\Vert^2_{L^2(\mathbf{1}_I\mu)}\leq 1.
$$
We first show that $\ddot{\mathcal{E}}$ is bounded. We have
\begin{eqnarray*}
\ddot{\mathrm{P}}(I,\ddot{\sigma})
=
\int\frac{|I|}{\left(|I|+|x-x_I|\right)^{3-\alpha}}d\ddot{\sigma}(x)
\!\!\!\!&\lesssim&\!\!\!\!
\sum_{n=0}^\infty \frac{\ddot{\sigma}\big((2^n+1)I\big)}{(2^n)|2^nI|^{2-\alpha}}\\
&\leq&\!\!\!\!
\sum_{n=0}^\infty \inf_{x \in I}M^\alpha\ddot{\sigma}(x)2^{-n}
\lesssim
\inf_{x \in I}M^\alpha\ddot{\sigma}(x)
\end{eqnarray*}
where 
$\displaystyle
M^\alpha\mu(x)=\sup_{I\ni x}\frac{1}{|I|^{2-\alpha}}\int_Id\mu
$
and the implied constants depend only on $\alpha$. Thus, given an interval $\displaystyle I=\dot{\cup}_{r\geq 1}I_r$, we have:
$$
\sum_{r \geq 1}\ddot{\omega}(I_r)\ddot{\mathrm{P}}^2(I_r,\mathbf{1}_{I}\ddot{\sigma})
\leq 
\sum_{r\geq 1}\ddot{\omega}(I_r) \inf_{x \in I}\left(M^\alpha\mathbf{1}_{I}\ddot{\sigma}\right)^2(x)
\leq 
\int_{I} \left(M^\alpha\mathbf{1}_{I}\ddot{\sigma}\right)^2(x)d\ddot{\omega}(x)
$$
and so we are left with estimating the right hand term of the above inequality. We will prove the inequality
\begin{equation}
\label{maximal}
\int_{I^\ell_r} \left(M^\alpha\mathbf{1}_{I^l_r}\ddot{\sigma}\right)^2(x)d\ddot{\omega}(x)\leq C \ddot{\sigma}(I^\ell_r).
\end{equation}
where the constant $C$ depends only on $\alpha$.
This will be enough, since for an interval $I$ containing a point mass $\ddot{z}^\ell_r$ but no masses $\ddot{z}^k_j$ for $k<\ell$, we have
\begin{eqnarray*}
 \int_{I} \left(M^\alpha\ddot{\sigma}\right)^2(x)d\ddot{\omega}(x)=
 \int_{I\cap I^\ell_r} \left(M^\alpha\mathbf{1}_{I\cap I^\ell_r}\ddot{\sigma}\right)^2(x)d\ddot{\omega}(x)
 &\leq&
 \int_{I^\ell_r}\left(M^\alpha\mathbf{1}_{I^\ell_r}\ddot{\sigma}\right)^2(x)d\ddot{\omega}(x)\\
 &\leq&
 \ddot{\sigma}(I^\ell_r)\approx \ddot{\sigma} (I)
\end{eqnarray*}
Since the measure $\ddot{\omega}$ is supported in the Cantor set $\mathrm{E}_b$, we can use the fact that for $x \in I^\ell_r\cap \mathrm{E}_b$,
$$
M^\alpha(\mathbf{1}_{I^\ell_r}\ddot{\sigma})(x)\lesssim\!\!\!
\sup_{\left( k,j\right) :x\in I_{j}^{k}}\frac{1}{\left\vert
I_{j}^{k}\right\vert^{2-\alpha} }\int_{I_{j}^{k}\cap I_{r}^{\ell }}\!\!\!\!\!d\ddot{\sigma} 
\approx \!\!\!\!\!
\sup_{\left( k,j\right) :x\in I_{j}^{k}}\!\!\!\!\frac{s_0
^{-2(k\vee \ell) } 2^{k\vee \ell}}{s_0
^{-k}}\approx \frac{\ddot{\sigma}(I^\ell_r)}{|I^\ell_r|^{2-\alpha}}\approx\left(\frac{2}{s_0}\right)^\ell
$$
Fix m and let the approximations $\ddot{\omega} ^{\left(
m\right) }$ and $\ddot{\sigma}^{\left( m\right) }$ to the measures $\omega $
and $\ddot{\sigma}$ given by 

\begin{eqnarray*}
d\ddot{\omega} ^{\left( m\right) }\left( x\right)
=
\sum_{i=1}^{2^{m}}2^{-m}\frac{1%
}{\left\vert I_{i}^{m}\right\vert }\mathbf{1}_{I_{i}^{m}}\left( x\right) dx
\ \text{  and  }\ 
\ddot{\sigma}^{\left( m\right) }
=
\sum_{k<m}\sum_{j=1}^{2^{k}}s_{j}^{k}\delta _{z_{j}^{k}}.  \notag
\end{eqnarray*}%
For these approximations we have in the same way the estimate for $ x\in \bigcup_{i=1}^{2^{m}}I_{i}^{m}$,
\begin{equation*}
M^\alpha\left( \mathbf{1}_{I_{r}^{\ell }}\ddot{\sigma}^{\left( m\right)
}\right) \left( x\right) 
\lesssim \!\!\!\!\!
\sup_{\left( k,j\right) :x\in I_{j}^{k}}\frac{1}{\left\vert
I_{j}^{k}\right\vert^{2-\alpha} }\int_{I_{j}^{k}\cap I_{r}^{\ell }}\!\!\!d\ddot{\sigma} \approx
\!\!\!\sup_{\left( k,j\right) :x\in I_{j}^{k}}\!\!\!\frac{\left( \frac{1}{s_0}\right)
^{k\vee \ell }\left( \frac{2}{s_0}\right) ^{k\vee \ell }}{\left( \frac{1}{s_0}%
\right) ^{k}}
\leq
C\left( \frac{2}{s_0}\right) ^{\ell }
\end{equation*}%
Thus for each $m\geq n \geq \ell$ we have%
\begin{eqnarray*}
\int_{I_{r}^{\ell }}\!\!\! M^\alpha\left( \mathbf{1}_{I_{r}^{\ell }}\ddot{\sigma}%
^{\left( n\right) }\right) ^{2}\!\!d\ddot{\omega} ^{\left( m\right) } 
\!\leq\!
C\!\!\!\sum_{i:I_{i}^{m}\subset I_{r}^{\ell }}\!\!\!\left( \frac{2}{s_0}\right)^{2\ell
}\!\!\!\!\!2^{-m}
\!=\!
C2^{m-\ell }\!\!\left( \frac{2}{s_0}\right) ^{2\ell }\!\!\!\!2^{-m}=Cs_{r}^{\ell
}\approx C\int_{I_{r}^{\ell }}d\ddot{\sigma}
\end{eqnarray*}%
Now since $\ddot{\omega}^m$ converges weakly to $\ddot{\omega}$ and using the fact that $M^\alpha$ is lower semi-continuous we get:
$$
\int_{I_{r}^{\ell }}\!\!\! M^\alpha\left( \mathbf{1}_{I_{r}^{\ell }}\ddot{\sigma}%
^{\left( n\right) }\right) ^{2}\!\!d\ddot{\omega} \leq  \liminf\limits_{m\rightarrow \infty} \int_{I_{r}^{\ell }}\!\!\! M^\alpha\left( \mathbf{1}_{I_{r}^{\ell }}\ddot{\sigma}%
^{\left( n\right) }\right) ^{2}\!\!d\ddot{\omega} ^{\left( m\right)} \leq C\ddot{\sigma}(I_{r}^{\ell })
$$
Now, taking $n\rightarrow \infty$, by monotone convergence we get (\ref{maximal}). This proves
\begin{equation}\label{pivotal}
 \sum_{r \geq 1}\ddot{\omega}(I_r)\ddot{\mathrm{P}}^2(I_r,\mathbf{1}_{I}\ddot{\sigma})\leq C\ddot{\sigma}(I)   \end{equation}
which in turn implies $\ddot{\mathcal{E}}<\infty$ as $E(I_r,\ddot{\omega})\leq 1$.

Finally, we show that the dual energy constant $\ddot{\mathcal{E}}^{*}$ is finite. Let us show that for $I\subset I_1^0$
\begin{equation}
\ddot{\sigma} (I)E(I,\ddot{\sigma} )^{2}\ddot{\mathrm{P}}(I,\ddot{\omega} )^{2}\lesssim
\ddot{\omega} (I).  \label{e.gettingE}
\end{equation}
as if we let $\{I_{r}\;:\;r\geq 1\}$ be any partition of $I$, (\ref{e.gettingE}) gives
\begin{equation*}
\sum_{r\geq 1}\ddot{\sigma} (I_{r})E(I_{r},\ddot{\sigma} )^{2}\ddot{\mathrm{P}}
(I_{r},\ddot{\omega} )^{2}\lesssim \sum_{r\geq 1}\ddot{\omega} (I_{r})=\ddot{\omega} (I_{})\ .
\end{equation*}

Now let us establish (\ref{e.gettingE}). We can assume that $E(I,\ddot{\sigma} )\neq 0$. Let $k$ be the smallest
integer for which there is a $r$ with $\ddot{z}_{r}^{k}\in I$. And let $n$
be the smallest integer so that for some $s$ we have $\ddot{z}_{s}^{k+n}\in
I $ and $\ddot{z}_{s}^{k+n}\neq \ddot{z}_{r}^{k}$. We have that 
\begin{eqnarray*}
E(I,\ddot{\sigma} )^2  &=& \frac{1}{2}\frac{1}{\ddot{\sigma}(I)^2}\int_I\int_I\frac{|x-x'|^2}{|I|^2}d\ddot{\sigma}(x)d\ddot{\sigma}(x')\\
&=&
\frac{1}{\ddot{\sigma}(I)^2}\left[\ddot{\sigma}(\ddot{z}_{r}^{k})\int_I\frac{|x-\ddot{z}_{r}^{k}|^2}{|I|^2}d\ddot{\sigma}(x)+\int_I\int_{I\backslash\{\ddot{z}_{r}^{k}\}}\frac{|x-x'|^2}{|I|^2}d\ddot{\sigma}(x)d\ddot{\sigma}(x')\right]\\
&\lesssim&
\frac{\ddot{\sigma}(\ddot{z}_{r}^{k})\ddot{\sigma}(I\backslash\{\ddot{z}_{r}^{k}\})}{\ddot{\sigma}(I)^2}+\frac{\ddot{\sigma}(I\backslash\{\ddot{z}_{r}^{k}\})}{\ddot{\sigma}(I)}\lesssim \left( \frac{2}{s_0^2}\right) ^{n}
\end{eqnarray*}
Finally, $\ddot{\sigma}( I)\approx \left( \frac{2}{s_0^2}\right) ^{k}$, $%
\ddot{\omega}( I)\approx 2^{-k-n}$, and $\ddot{\mathrm{P}}(I,\ddot{\omega} )\approx \left( \frac{s_0}{2}%
\right) ^{k}$, which proves \eqref{e.gettingE}.

$$
\underline{\textsc{The Two Dimensional Construction}}
$$
It is time now to define the two dimensional measures that prove the statement of Theorem \ref{fracint}. For any set $E\subset \mathbb{R}^2$ let
$$
\omega(E)=\sum_{n=0}^\infty \ddot{\omega}_n(E)
$$
where $\ddot{\omega}_0(E)=\ddot{\omega}(E_x\cap I_1^0)$, $E_x$ the projection of $E$ on the x-axis, and $\ddot{\omega}_n$ are copies of $\ddot{\omega}_0$ at the intervals $[a_n,a_n+1]\times \{0\}$ with $k_n=a_{n+1}-(a_n+1)$ to be determined later. In the same way, let
$$
\sigma(E)=\sum_{n=0}^\infty \ddot{\sigma}_n(E)
$$
where $\ddot{\sigma}_0(E)=\ddot{\sigma}([E\cap (I_1^0\times\{\gamma_0\})]_x)$, and $\ddot{\sigma}_n$ are copies of $\ddot{\sigma}_0$ at the intervals $[a_n,a_n+1]\times \{\gamma_n\}$, where the height $\gamma_n$ will be determined later.
\begin{center}
\begin{tikzpicture}
\label{Figure 2.1.}
\color{blue}
\draw (-6,0) -- (-4,0);
\draw (-3,0) -- (-1,0);
\draw (0.3,0) -- (2.3,0);
\draw (4.05,0) -- (6.05,0);
\node (f) at (-5,-0.2) {$\omega$};
\node (g) at (-2,-0.2) {$\omega$};
\node (h) at (1.3,-0.2) {$\omega$};
\node (i) at (5.05,-0.2) {$\omega$};
\color{red}
\draw (-6,1) -- (-4,1);
\draw (-3,0.5) -- (-1,0.5);
\draw (0.3,0.2) -- (2.3,0.2);
\draw (4.05,0.1) -- (6.05,0.1);
\node (a) at (-5,1.2) {$\sigma$};
\node (b) at (-2,0.7) {$\sigma$};
\node (c) at (1.3,0.4) {$\sigma$};
\node (d) at (5.05,0.3) {$\sigma$};
\color{black}
\draw[decorate,decoration={brace}]  (-6.1,0) -- (-6.1,1)
node (k) at (-6.4,0.5){\footnotesize $\gamma_n$};
\draw[decorate,decoration={brace,mirror}]  (-4,-0.2) -- (-3,-0.2)
node (m) at (-3.5,-0.4){\footnotesize $k_n$};
\end{tikzpicture}\vspace{-0.3cm}
Figure 1
\end{center}
\textsc{The $\mathcal{A}_2$ conditions.}
We will now prove that both $\mathcal{A}^{\alpha}_2$ and $\mathcal{A}^{\alpha,*}_2$ constants are bounded. Let $Q$ be a cube in $\mathbb{R}^2$, $J^n_0=[a_n,a_n+1]\times \{0\}$ and $J^n_{\gamma_n}=[a_n,a_n+1]\times \{\gamma_n\}$. We take cases for $Q$. If $Q$ intersects only one of the intervals $J^n_0$, say $J^0_0$ for convenience, and $(Q\cap J^0_0)_x=:J_0$ we have:
\begin{eqnarray*}
\mathcal{P}^\alpha(Q,\textbf{1}_{Q^c}\sigma)\frac{\omega(Q)}{|Q|^{1-\frac{\alpha}{2}}}
&\lesssim &
\ddot{\mathcal{P}}(J_0,\ddot{\sigma})\frac{\ddot{\omega}(J_0)}{|J_0|^{2-\alpha}}
+
\mathcal{P}^\alpha(Q,\textbf{1}_{(J^1_{\gamma_1})^c}\sigma)\frac{\ddot{\omega}(I_1^0)}{|Q|^{1-\frac{\alpha}{2}}}\\
&\leq&
\ddot{\mathcal{A}}_2^\alpha(\ddot{\sigma},\ddot{\omega})+C<\infty
\end{eqnarray*}
using (\ref{var_a2_bdd}) and taking $k_n$ large enough so that the second summand is bounded independently of the interval ($k_n=4^{2n\cdot \max\{(2-\alpha)^{-1},1\}}$ would do here). If $Q$ intersects more than one of the intervals $J^n_0$, it is easy to see, using that $Q$ is very big (since it intersects more than one of the intervals) and that $k_n$ is also large,  that:
\begin{eqnarray*}
\mathcal{P}^\alpha(Q,\textbf{1}_{Q^c}\sigma)\frac{\omega(Q)}{|Q|^{1-\frac{\alpha}{2}}}
\lesssim
1
\end{eqnarray*}
which of course shows that $\mathcal{A}^{\alpha}_2$ is bounded.
Essentially using the same calculations we see that $\mathcal{A}^{\alpha,*}_2$ is bounded as well. \\ $ $\\
\textsc{Off-Testing Constant}. Let us now check that the off-testing constant is not bounded. Choose the cube $Q_n=[a_n,a_n+1]\times[0,-1]$. Then, 
\begin{eqnarray*}
\frac{1}{\omega (Q_n)}\!\int_{Q_n^c}\!\!\bigg[\!\int_{Q_n}\!\frac{d\omega (y)}{|x-y|^{2-\alpha}}\bigg]^2\!\!d\sigma (x)
\!\geq\!
\frac{1}{\ddot{\omega}(I_1^0)}\!\int_{I_1^0}\!\!\bigg[\!\int_{I_1^0}\!\frac{d\ddot{\omega} (y_1)}{\sqrt{(x_1-y_1)^2+\gamma_n^2}^{2-\alpha}}\bigg]^2\!\!d\ddot{\sigma}(x_{1}\!)
\end{eqnarray*}
for $x=(x_1,x_2)$ and $y=(y_1,y_2)$. Taking $\gamma_n$ such that the last expression on the display above equals $n$ (note that this is feasible, since for $\gamma_n=0$, \eqref{testinfinity} gives infinity in the latter expression above) we have
$$
\mathcal{T}_{\textit{off},\alpha}^2\geq \frac{1}{\omega (Q_n)}\int_{Q_n^c}\bigg[\int_{Q_n}\frac{d\omega (y)}{|x-y|^{2-\alpha}}\bigg]^2d\sigma (x)\geq n
$$
and by letting $n\rightarrow \infty$ we obtain that the off-testing constant is not bounded. \\ $ $\\
\textsc{The Energy Conditions}. 
For the energy condition $\mathcal{E}_2^\alpha$ first, let $Q$ be a cube and $Q=\dot{\cup}Q_r$, where $\{Q_r\}_{r=1}^\infty$ is a decomposition of Q. Then we have 
$$
\frac{1}{\sigma(Q)}\sum_{r=1}^{\infty }\left( \frac{\mathrm{P}
^{\alpha }\left( Q_{r},\mathbf{1}_{Q}\sigma \right) }{\left\vert
Q_{r}\right\vert ^{\frac{1}{2}}}\right) ^{2}\left\Vert
x-m^\omega_{Q_{r}}\right\Vert _{L^{2}\left( \mathbf{1}_{Q_{r}}\omega \right) }^{2}
\!\!\!\leq
\frac{2}{\sigma(Q)}\sum_{r=1}^\infty \omega(Q_r)\big(\mathrm{P}^{\alpha }\left( Q_{r},\mathbf{1}_{Q}\sigma \right)\big)^2
$$
Assume  that $Q$ intersects $m$ intervals of the form $J_0^n$. Then we have $m-2\lesssim \sigma(Q)\lesssim m$. The case $m=1$ is exactly the same as the one dimensional analog for $\ddot{\mathcal{E}}$. Assume $m=2$. Now we need to take cases for $Q_r$:
\begin{enumerate}[(i)]
    \item Let $Q^1$ be the set of cubes $Q_r$ that intersect only one of the intervals $J^n_0$. Then we have, following the proof of \eqref{pivotal}, that 
    \begin{equation*}
    \sum_{Q_r \in Q^1}\omega(Q_r)\left(\mathrm{P}^{\alpha }\left( Q_{r},\mathbf{1}_{Q}\sigma \right)\right)^2\leq C\sigma(Q)
    \end{equation*}
    
    \item If $Q_r$ intersects both of the intervals $J^n_0$ then this $Q_r$ is unique since the family $\{Q_r\}_{r \in \mathbb{N}}$ forms a decomposition of $Q$. Therefore we have:
     \begin{equation*}
    \omega(Q_r)\left(\mathrm{P}^{\alpha }\left( Q_{r},\mathbf{1}_{Q}\sigma \right)\right)^2\lesssim  \frac{\omega(Q_r)\sigma(Q)}{|Q_r|^{2-\alpha}}\sigma(Q)\lesssim \sigma(Q)
    \end{equation*}
    using the fact that $|Q_r|\gtrsim 4^2$   since it intersect two of the intervals $J^n_0$ and $\omega(Q_r)\lesssim 2, \sigma(Q)\lesssim 2$.
\end{enumerate}
For $m\geq 3$, again we take cases for $Q_r$:
\begin{enumerate}[(i)]
    \item If $Q_r$ intersects only one $J^n_0$ we again have, following the proof of \eqref{pivotal}, that
   \begin{equation*}
    \sum_{Q_r \in Q^1}\omega(Q_r)\left(\mathrm{P}^{\alpha }\left( Q_{r},\mathbf{1}_{Q}\sigma \right)\right)^2\leq C\sigma(Q)
    \end{equation*}
    \item If $Q_r$ intersects more than one of the intervals $J^n_0$, the last one being $J^{n_0}_0$ we have
     \begin{equation*}
    \omega(Q_r)\left(\mathrm{P}^{\alpha }\left( Q_{r},\mathbf{1}_{Q}\sigma \right)\right)^2\lesssim \frac{\omega(Q_r)\sigma(Q_r^-)^2}{|Q_r|^{2-\alpha}}+\omega(Q_r)\sum_{k=1}^m\frac{1}{4^{2k}|Q_r|^{2-\alpha}} \lesssim  2
    \end{equation*}
    where $Q_r^-$ contains all the intervals $J^n_0$ such that  $n \leq n_0$. Again in the last inequality we use the fact that $Q_r$ is very big since it intersects at least two intervals $J^n_0$. Now since $Q_r$ form a decomposition of $Q$ we can have at most $m-1$ of these.
\end{enumerate}
Combining the above cases, we obtain
$$
\sum_{r=1}^\infty \omega(Q_r)\big(\mathrm{P}^{\alpha }\left( Q_{r},\mathbf{1}_{Q}\sigma \right)\big)^2\leq C\sigma(Q)+2m-2\leq 2C\sigma(Q)
$$
and that proves the energy condition is bounded.

The dual energy $\mathcal{E}^{\alpha,*}_2$ can also be proved bounded with the same calculations as in the energy condition following the proof of \eqref{e.gettingE}  instead of \eqref{pivotal} as in the first case above. This completes the proof of the Theorem \ref{fracint}.
\end{proof}
To obtain the same result for the Riesz transforms, we need to deal with the fact that the kernel is not positive. This prevents us from placing the masses for $\ddot{\sigma}$ at the center of the intervals $G^k_j$, as we did in the proof of Theorem \ref{fracint}.  Since otherwise, if the point-mass $\ddot{\sigma}$ is located at the center of $G^k_j$, it would result in the cancellation of much of the mass not letting us deduce that the off testing condition for the Riesz transform is unbounded. The following lemma, whose proof follows closely the work in \cite{LaSaUr} but with a two dimensional twist, helps us overcome this problem, showing that, while not being able to place the point masses in the middle of $G^k_j$, we can place them far from the boundary. This enables us to show that the $\ddot{\mathcal{A}}_2$ condition is bounded, like in the proof of Theorem \ref{fracint}. First we need to define the operator
$$
\ddot{R}f(x)=\int_\mathbb{R}\frac{(x-y)f(y)}{|x-y|^{3-\alpha}}dy
$$
\begin{lemma}\label{lemma}
For $k \geq 1,\ 1\leq j \leq 2^k$, write $G^k_j=(a^k_j,b^k_j)$. Then there exists $0<c<1$ that depends only on $\alpha$ such that
$$
\ddot{R}\ddot{\omega}\!\left(\!a^k_j\!+\!c\left(\!\frac{1-b}{2}\right)^{k}\!b\!\right)
\approx
\left(\frac{s_0}{2}\right)^k
$$
where $\ddot{\omega}$ is the measure defined above.
\end{lemma}
\begin{proof}
Fix $k$. We have 
$$
\ddot{R}\ddot{\omega}\!\left(\!a^k_1\!+\!c\left(\!\frac{1-b}{2}\right)^{k}\!b\!\right)
\!\leq\!
\ddot{R}\ddot{\omega}\!\left(\!a^k_j\!+\!c\left(\!\frac{1-b}{2}\right)^{k}\!b\!\right)
\!\leq\!
\ddot{R}\ddot{\omega}\!\left(\!a^k_{2^k}\!+\!c\left(\!\frac{1-b}{2}\right)^{k}\!b\!\right)
$$
from monotonicity. So it is enough to prove the following:
$$
\left(\frac{s_0}{2}\right)^k
\lesssim
\ddot{R}\ddot{\omega}\!\left(\!a^k_1\!+\!c\left(\!\frac{1-b}{2}\right)^{k}\!b\!\right)
\leq
\ddot{R}\ddot{\omega}\!\left(\!a^k_{2^k}\!+\!c\left(\!\frac{1-b}{2}\right)^{k}\!b\!\right)
\lesssim 
\left(\frac{s_0}{2}\right)^k 
$$
We start with right hand inequality. Following the definitions of $\ddot{R},\ddot{\omega}$ we get
\begin{eqnarray*}
\ddot{R}\ddot{\omega}\!\left(\!a^k_{2^k}\!+\!c\left(\!\frac{1-b}{2}\right)^{k}\!b\!\right)
\!\!\!\!&\leq&\!\!\!\!
\int_{[0,a^k_{2^k}]}\frac{d\ddot{\omega}(y)}{\left(a^k_{2^k}\!+\!c\left(\frac{1-b}{2}\right)^{k}\!b\!-\!y\right)^{2-\alpha}}\\
\!\!\!\!&\leq&\!\!\!\!
\sum_{\ell=1}^{k}\frac{2^{-\ell}}{\left(a^k_{2^k}\!+\!c\left(\frac{1-b}{2}\right)^{k}\!b\!-\!\left[1\!-\!\left(\frac{1-b}{2}\right)^{\ell-1}\!\left(\frac{1+b}{2}\right)\right]\right)^{2-\alpha}}\\
&\approx&\!\!\!\!
\frac{2^{-k}}{c^{2-\alpha}s_0^{-k}}\!+\!\sum_{\ell=1}^{k-1}\frac{2^{-\ell}}{s_0^{-\ell}\!\left[\frac{1+b}{2}\!+\!\left(\frac{1-b}{2}\right)^{k-\ell+1}\!\left[cb\!-\!\frac{1+b}{2}\right]\right]^{2-\alpha}}\\
&\leq&\!\!\!\!
\frac{2^{-k}}{c^{2-\alpha}s_0^{-k}}\!+\!\sum_{\ell=1}^{k-1}\frac{2^{-\ell}}{s_0^{-\ell}\!\left[\frac{1+b}{2}\!-\!\frac{1+b}{2}\left(\frac{1-b}{2}\right)^{k-\ell+1}\right]^{2-\alpha}}
\end{eqnarray*}
since $a_{2^k}^k=1\!-\!\left(\frac{1+b}{2}\right)\left(\frac{1-b}{2}\right)^k$. The square bracket inside the last fraction is minimized for $\ell=k-1$ and we get the inequality
$$
\ddot{R}\ddot{\omega}\!\left(\!a^k_{2^k}\!+\!c\left(\!\frac{1-b}{2}\right)^{k}\!b\!\right)
\lesssim
\frac{2^{-k}}{c^{2-\alpha}s_0^{-k}}+\sum_{\ell=1}^{k-1}\left(\frac{s_0}{2}\right)^\ell \lesssim \frac{1}{c^{2-\alpha}}\left(\frac{s_0}{2}\right)^k
$$
where the implied constants depend again only on $\alpha$. We should note here that the summand with $\ell=k$ is the dominant one in the above inequality. 

Now we consider the left hand inequality. We have that
$
\ddot{R}\ddot{\omega}\left(\!a^k_1\!+\!c\!\left(\!\frac{1-b}{2}\!\right)^{\!k}\!b\right)
$ equals
\begin{equation}\label{varR}
\ddot{R}\ddot{\omega}\mathbf{1}_{I^{k+1}_1}
\left(\!a^k_1\!+\!c\!\left(\!\frac{1-b}{2}\!\right)^{\!k}\!b\right)
+\sum_{\ell=1}^{k+1}\ddot{R}\ddot{\omega}\mathbf{1}_{I^\ell_2}\left(\!a^k_1\!+\!c\!\left(\!\frac{1-b}{2}\!\right)^{\!k}\!b\right)
\end{equation}
and following the argument for the previous inequality we see that
$$
\left|\sum_{\ell=1}^{k+1}\ddot{R}\ddot{\omega}\mathbf{1}_{I^\ell_2}\left(\!a^k_1\!+\!c\!\left(\!\frac{1-b}{2}\!\right)^{\!k}\!b\right)\right|\leq A\left(\frac{s_0}{2}\right)^k
$$
where $A$ depends only on $\alpha$ but not on $c$. The first summand of (\ref{varR}) gives
\begin{eqnarray*}
\int_{I_1^{k+1}}\frac{d\ddot{\omega}(y)}{\left(a^k_1\!+\!c\left(\frac{1-b}{2}\right)^{k}\!b\!-\!y\right)^{2-\alpha}}
&\geq&
\sum_{\ell=k+1}^\infty\frac{2^{-\ell-1}}{\left(\left(\frac{1-b}{2}\right)^\ell+c\left(\frac{1-b}{2}\right)^{k}b\right)^{2-\alpha}}
\end{eqnarray*}
\begin{eqnarray*}
&\hspace{4.3cm}\approx&
\frac{s_0^k}{2^k}\sum_{\ell=k+1}^\infty\frac{2^{-\ell+k-1}}{\left(\left(\frac{1-b}{2}\right)^{\ell-k}\!+\!cb\right)^{2-\alpha}}\\
&\hspace{4.3cm}=&
\frac{s_0^k}{2^k}\sum_{\ell=1}^\infty\frac{2^{-\ell-1}}{\left(\left(\frac{1-b}{2}\right)^{\ell}\!+\!c b\right)^{2-\alpha}}.
\end{eqnarray*}
Choosing $c$ small enough not depending on $k$ (since the last sum does not depend on $k$), we obtain 
$$
\int_{I_1^{k+1}}\frac{d\ddot{\omega}(y)}{\left(a^k_1\!+\!c\left(\frac{1-b}{2}\right)^{k}\!b\!-\!y\right)^{2-\alpha}}  \geq C_1\left(\frac{s_0}{2}\right)^k
$$
with $C_1>2A$ and we conclude our lemma.
\end{proof}
\begin{proof}[Proof of Theorem \ref{Riesz}]
Set $\dot{z}^k_j=a^k_j\!+\!cb\left(\!\frac{1-b}{2}\right)^{k}$ and define the measure 
$\displaystyle
\dot{\sigma}=\sum_{k,j}s^k_j\delta_{\dot{z}^k_j}
$
where $s_j^k=\left(\frac{2}{s_0^2}\right)^k$ as before. Following verbatim the calculations of Theorem \ref{fracint}, one can show that
$
\ddot{\mathcal{A}}_2(\dot{\sigma},\ddot{\omega})<\infty.
$
Now define the measures $\omega$ and $\sigma$, as before, for any measurable set $E\subset\mathbb{\R}^2$ by

$$
\omega(E)=\sum_{n=0}^\infty \ddot{\omega}_n(E)\ \ \text{and}\ \ \sigma(E)=\sum_{n=0}^\infty \dot{\sigma}_n(E)
$$
where $\dot{\sigma}_0(E)=\dot{\sigma}([E\cap (I_1^0\times\{\gamma_0\})]_x)$, and $\dot{\sigma}_n$ are copies of $\dot{\sigma}_0$ at the intervals $[a_n,a_n+1]\times \{\gamma_n\}$, and where the height $\gamma_n$ will be determined later. Again, as before, it is easy to see that both $\mathcal{A}_2^\alpha$ and $\mathcal{A}_2^{\alpha,*}$ and both $\mathcal{E}^\alpha_2$ and $\mathcal{E}^{\alpha,*}_2$ are bounded. Let us now finish the proof by showing that the off-testing constant for the Riesz transforms are unbounded. From Lemma \ref{lemma} we have  
$
\ddot{R}\ddot{\omega}(\dot{z}_j^k)\gtrsim \left(\frac{s_0}{2}\right)^k
$
which implies
\begin{eqnarray}
\int_{I_1^0}\left(\ddot{R}(\mathbf{1}_{I_1^0}\ddot{\omega})(x)\right)^2d\dot{\sigma}(y)
\gtrsim
\sum_{k=1}^
\infty\sum_{j=1}^{2^k}s^k_j\cdot \left(\frac{s_0}{2}\right)^{2k}
= \label{test_R_infty}
\sum_{k=1}^\infty\sum_{j=1}^{2^k}\frac{1}{2^k}=\infty. 
\end{eqnarray}
Now choose the cube $Q_n=[a_n,a_n+1]\times[0,-1]$. Then,
\begin{eqnarray*}
\mathcal{R}_{1,\textit{off},\alpha}^2
&\geq&
\frac{1}{\omega (Q_n)}\int_{Q_n^c}\bigg[\int_{Q_n}\!\frac{(x_1-y_1)d\omega (y)}{|x-y|^{3-\alpha}}\bigg]^2\!d\sigma (x)\\
&\geq&
\frac{1}{\omega(Q_n)} \int_{I_1^0}\bigg[\int_{I_1^0}\frac{(x_1-y_1)d\ddot{\omega} (y_1)}{\sqrt{(x_1-y_1)^2+\gamma_n^2}^{3-\alpha}}\bigg]^2\!d\dot{\sigma}(x_1)
=
\frac{n}{\omega(Q_n)}
\end{eqnarray*}
by choosing the height $\gamma_n$ so that
$\int_{I_1^0}\!\!\bigg[\!\int_{I_1^0}\frac{(x_1-y_1)d\ddot{\omega} (y_1)}{\sqrt{(x_1-y_1)^2+\gamma_n^2}^{3-\alpha}}\bigg]^2\!d\dot{\sigma}(x_1)=n$ by (\ref{test_R_infty}). Letting $n\rightarrow\infty$, we see that the off-testing constant is unbounded.
\end{proof}
\textbf{Acknowledgement:} We would like to thank professors Eric Sawyer and Ignacio Uriarte-Tuero for edifying discussions and their useful suggestions.


\begin{thebibliography}{40}
\addtolength{\leftmargin}{-0.5in} 
\setlength{\itemindent}{-0.5in}

\bibitem[La]{La} \hspace{1.1cm} Lacey M. \textit{Two weight inequality for the Hilbert transform: A real variable characterization, II}, Duke Math. J. Volume 163, Number 15 (2014), 2821-2840.

\bibitem[LaSaShUT]{LaSaShUT} \hspace{0.13cm}Lacey M., Sawyer E., Shen C-S, Uriarte-Tuero I. \textit{Two weight inequality for the Hilbert transform: A real variable characterization, I}, Duke Math. J. Volume 163, Number 15 (2014), 2795-2820.

\bibitem[LaSaUr]{LaSaUr} \hspace{0.48cm} \textsc{Lacey M., Sawyer E.,
Uriarte-Tuero I., } \textit{A Two Weight Inequality for the Hilbert
transform assuming an energy hypothesis, } Journal of Functional Analysis,
Volume \textbf{263} (2012), Issue 2, 305-363.

\bibitem[Hy]{Hy} \hspace{1.1cm} Hyt\"{o}nen, T. \textit{The two-weight inequality for the Hilbert transform with general measures}, Proc. London Math. Soc. (3) 117 (2018) 483-526.

\bibitem[HyMa]{HyMa} \hspace{0.65cm} Hyt\"{o}nen, T., Martikainen, H.  \textit{On general local $T_b$ theorems}. Transactions of the American Mathematical Society, 364(9), (2012), 4819-4846

\bibitem[SaShUT]{SaShUT}\hspace{0.35cm} Sawyer, E., Shen, C-S, Uriarte-Tuero, I. \textit{A two weight local $T_b$ Theorem for the Hilbert Transform}, arXiv:1709.09595.

\end{thebibliography}
\end{document}